\documentclass[12pt]{extarticle}
\usepackage{amsmath,amssymb, amsthm}
\usepackage{geometry,comment}
\usepackage{tikz-cd}

\newtheorem{theorem}{Theorem}[section]
\newtheorem{lemma}[theorem]{Lemma}
\newtheorem{corollary}[theorem]{Corollary}

\newtheorem{definition}[theorem]{Definition}
\newtheorem{proposition}[theorem]{Proposition}

\def\mb{\mathbb}
\def\mc{\mathcal}

\def\L{\Lambda}
\def\k{\mb K}

\def\C{\mb C}

\def\Z{\mb Z}

\def\Q{\mb Q}
\def\ph{\varphi}

\def\oL{\ol L}
\def\t{\times}
\def\cl{\colon}

\def\ol{\overline}
\def\wt{\widetilde}

\renewcommand{\P}{{\mb P^1}}
\def\wdw{\wedge\dots\wedge}

\def\bs{\backslash}

\newcommand{\res}[2]{\left.{#1}\right|_{#2}}

\DeclareMathOperator{\ts}{\partial}
\DeclareMathOperator{\ord}{ord}

\DeclareMathOperator{\im}{Im}

\DeclareMathOperator{\spec}{Spec}

\DeclareMathOperator{\supp}{supp}

\begin{document}
\title{On functional equations for Chow polylogarithms}
\author{Vasily Bolbachan\footnote{This paper is an output of a research project implemented as part of the Basic Research Program at the National Research University Higher School of Economics (HSE University). This paper was also supported in part by the contest "Young Russian Mathematics"}}

\maketitle
\begin{abstract}
Chow polylogarithms are some special functions arising in explicit description of the Beilinson regulator map. The most interesting functional equation for this function reflects its vanishing on the boundary in the Bloch's cycle complex. We show that this functional equation formally follows from more simple ones, namely skew-symmetry, functoriality and multiplicativity. 

%We study functional equations for Chow polylogarithms and prove that most of them formally follows from skew-symmetry, multiplicativity and functoriality.

To prove this, we study some analogue of Bloch's cycle complex and establish for this complex an analogue Beilinson-Soule vanishing conjecture.

A. Goncharov defined a group of functional equations for classical polylogarithms. We show that any such functional equation formally follows from functional equations for Chow polylogarithms stated above. 

\end{abstract}

\section{Introduction}

Let $X$ be a smooth proper variety over $\C$.  A. Goncharov \cite{goncharov2005polylogarithms} constructed an explicit map from the motivic cohomology of $X$ to its Deligne cohomology  (see also \cite{gil2011goncharov}). In the simplest case $X=\spec \C$ this construction is reduced to so-called Chow polylogarithms. Let $Y$ be $(m-1)$-dimensional variety. We write $\C(Y)^\t$ for the multiplicative group of non-zero rational functions on $Y$. Define a map $\theta_Y\cl \L^{2m-1}\C(Y)^\t\to\mb C$ by the following formula
$$\theta_Y(f_1\wdw f_{2m-1})=\int_{Y(\C)}\log |f_1|d\log |f_2|\wdw d\log|f_{2m-1}|.$$
\emph{The Chow polylogarithm} $\mc P_m$ is defined by the formula
$\mc P_m(Y,a)=\theta_Y(a)$. Up to some simple constant \cite{goncharov2001grassmannian}, this definition coincides with the definition given in \cite{goncharov2005polylogarithms}. 

Our immediate goal is to formulate functional equations for this function and investigate  relationships between them.

\emph{An alteration} is a proper morphism which is generically finite. Let $Y_1, Y_2$ be two proper varieties and $\ph\cl Y_1\to Y_2$ be an alteration. The following formula follows directly from the definition:
\begin{equation*}
    \tag{*}
    \mc P_m(Y_2,a)=\dfrac 1{\deg\ph}\mc P_m(Y_1,\ph^*(a)).
\end{equation*}

Let $(F,\nu)$ be a discrete valuation field. We denote by $\ol F_\nu$ the residue field. The following proposition was proved in \cite{goncharov1995geometry}.

\begin{proposition}
    Let $k\geq 1$. There is a unique map $\ts_\nu\cl\L^k F^\t\to \L^{k-1}\overline{F}_\nu^{\t}$ such that the following property holds. Let $a_1,\dots, a_k\in F^\t$ and assume that $\nu(a_2)=\dots=\nu(a_k)$. Then $\ts_D(a_1\wdw a_k)=\nu(a_1)\cdot(\overline{a_2}\wdw \overline{a_k})$.
\end{proposition}

Let $Y$ be a smooth variety and $D$ be an irreducible  divisor. Denote by $\nu_D$ the map $\ts_{\nu_D}$ where $\nu_D$ is a discrete valuation corresponding to $D$. Let us formulate another functional equation. Let $Y$ be a proper $m$-dimensional variety and $f_1,\dots, f_{2m}$ are non-zero rational functions on $Y$. Assume that the union of zeros and poles of the functions $f_i$ is a snc divisor on $Y$. Set $a=f_1\wdw f_{2m}$. Then

\begin{equation}
  \tag{**}
  \sum\limits_{D\subset Y}\mc P_m(a,\ts_D(a))=0.
\end{equation}

 To investigate the relationship between $(*)$ and $(**)$ we need the following definition.

\begin{definition}
    Denote by $\widehat C_m(\k)$ a vector space over $\Q$, freely generated by isomorphism classes of the pairs $[Y,a]$, where $Y$ is smooth projective variety of dimension $m-1$ and $a\in\L^{2m-1}\k(Y)^\t$. Denote by $\widehat B_m(\k)$ the quotient of $\widehat C_m(\k)$ by the following relations:
    \begin{enumerate}
        \item Let $\ph\cl Y_1\to Y_2$ be an alteration and $a\in\L^{2m-1}(\k(Y_2)^\t)$. Then
        $$[Y_2,a]=\dfrac 1{\deg\ph}[Y_1,\ph^*(a)].$$
        \item For any $a,b\in \L^{2m-1}(\k(Y)^\t)$, we have
         $$[Y,a+b]=[Y,a]+[Y,b].$$
    \end{enumerate}
\end{definition}

 The Chow polylogarithm $\mc P_m$ provides the canonical map $\widehat B_m(\C)\to\C$ given by the formula $[Y,a]\mapsto \mc P_m(Y,a)$. This implies that the group $\widehat B_m(\C)$ is non-zero. 

\begin{theorem}
\label{th:reciprocity_law}
    Let $m\geq 1$ and $Y$ be a smooth proper $m$-dimensional variety. Let $f_1,\dots, f_{2m}$ be non-zero rational functions on $Y$ and assume that the union of zeros and poles of the functions $f_i$ is a snc divisor on $Y$. Set $a=f_1\wdw f_{2m}$. Then the following formula holds:
    $$\sum\limits_{D\subset X}[D,\ts_D(a)]=0\in \widehat B_m(\k).$$
\end{theorem}

In other words, $(*)$ formally implies $(**)$. The case $m=1$ gives classical Weil reciprocity law. This statement is a natural generalization of Weil reciprocity law to a variety of arbitrary dimension. We believe that the group $\widehat B_m(\k)$ has the same relationship to Chow polylogarithms as the classical Bloch group \cite{goncharov1995geometry} $B_m(\k)$ has to classical polylogarithms.

Let $P\in\k[x_1,\dots, x_{m-1}]$. Define the following element of the group $\widehat B_m(\k)$:
$$\gamma_P=[(\P)^{m-1}, (1-x_1)\wedge x_1\wdw (1-x_{m-1})\wedge x_{m-1}\wedge P(x_1,\dots, x_{m-1})].$$ The following statement follows from the proof of Theorem \ref{th:reciprocity_law}.

\begin{theorem}
\label{th:generators_of_Bloch_group}
    The group $\widehat B_m(\k)$ is generated by the elements of the form $\gamma_P$.
\end{theorem}
\subsection{The complex $\L$ and Beilinson-Soule vanishing conjecture}
The definition of cubical higher Chow groups can be found in \cite{levine1994bloch}. Let $\square^n=(\P\bs\{0\})^n$. Subvarieties given by the equations of the form $z_i=0,\infty$ are called \emph{faces}. The group $z^m(n)$ is some subquotient of  codimension $m$ algebraic cycles on $\square^n$ which are in good position with respect to faces. The graded abelian group $z^m(2m-*)$ forms a chain complex where the differential is given by the sign sum of restrictions to codimension $1$ faces. It is known that the cohomology of this complex is isomorphic to the motivic cohomology of $\k$.

To represent an algebraic cycle in $\square^n$ we can take an algebraic variety $Y$ of dimension $n-m$ together with $n$ non-zero rational functions $f_1,\dots, f_n$. The cycle is $\psi_*([Y])\cap\square^n$, where the map $\psi\cl X\to (\P)^n$ is given by the formula $\psi(z)=(f_1(z),\dots, f_n(z))$. The idea of \cite{bolbachan2023chow,bol_2024} was to replace the data $(Y, f_1,\dots, f_n)$ by the data $(Y, a)$ where $a\in \L^n(\k(Y)^\t)$. The differential is given by the sum of the residue maps.

Let $X$ be an algebraic variety and $m, j\in\Z$. Set $p=\dim X+m-j, n=2m-j$. Denote by $\L(X,m)_j$ a vector space over $\Q$, freely generated by the triples $[Y,a,\psi]$, where: $Y$ is a variety of dimension $p$, $a\in\L^n\k(Y)^\t$ and $\psi\cl Y\to X$ is a proper morphism. The relations has the following form:
   \begin{enumerate}
        \item Let $[Y_2,a,\psi]\in \L(X,m)_j$ and $\ph\cl Y_1\to Y_2$ be an alteration. Then
        $$[Y_2,a,\psi]=\dfrac 1{\deg\ph}[Y_1,\ph^*(a),\psi\circ\ph].$$
        \item Let $[Y,a,\psi], [Y,b,\psi]\in\L(X,m)_j$. Then
         $$[Y,a+b,\psi]=[Y,a,\psi]+[Y,b,\psi].$$
    \end{enumerate} 
    In the case $X=\spec\k$ we write $[Y,a]$ instead of $[Y,a,\psi]$ and denote this complex by $\L(\k,m)$.
%In particular for $Y=\spec\k$ and an automorphism $\alpha$ of $Y$ we get
%$$[X,a]=[X,\alpha^*(a)].$$
 The differential in the complex $\L(X,m)$ is defined as follows. Let $f_1,\dots, f_n\in \k(Y)^\t$ and $a=f_1\wdw f_n$. Taking resolution of singularities, if necessary, we can assume that the union of zeros and poles of $f_i$ is a snc divisor.  By definition
$$d([Y,a,\psi])=\sum\limits_{D\subset Y}[D,\ts_D(a),\psi\circ i_D].$$
In this formula the sum is taken over all irreducible divisors and $i_D$ is the canonical embedding $D$ into $Y$. It was checked in \cite{bol_2024} that $d$ is well defined and $d^2=0$.

The complex $\L(X,m)$ is in many respects similar to the complex calculating higher Chow groups. We conjectured that these two complexes have the same cohomology. While it is not known in full generality, we proved this statement in the case when $X$ is a spectrum of a field and the degree is greater than or equal to the motivic weight minus one. It is natural to ask whether the complex $\L(X,m)$ satisfies an analog of the Beilinson-Soule vanishing conjecture. Standard reductions show that we can assume that $X=\spec\k$.

\begin{theorem}
\label{th:BS_Lambda}
    We have $\L(\k,m)_j$ is zero for $j\leq 0$ and $m\geq 1$.
\end{theorem}
So, surprisingly, the vanishing holds degree-wise. This statement gives some evidence for the original Beilinson-Soule vanishing conjecture. The Theorem \ref{th:reciprocity_law} is a reformulation of the statement that the differential map $d\cl \L(\k,m)_0\to \L(\k,m)_1$ is zero.

\subsection{Application to classical polylogarithms}

The classical $m$-logarithm is defined by the formula $$Li_m(z)=\sum\limits_{k=1}^{\infty}\dfrac{z^k}{k^m}.$$
We need a certain single-valued version of this function $\wt{\mc L_m}$ defined in \cite{Levin_A}.
%The single-valued version of this function \cite{zagier1991polylogarithms} is defined by the formula:
%$$\mc L_m(z)=\pi_m\left(\sum\limits_{r=0}^{m-1}\dfrac{2^kB_k}{k!}\log^k|z|Li_{k-r}(z)\right).$$
%In this formula $B_k$ are Bernoulli numbers. We need some modification of this function \cite{Levin_A}:
%$$\wt{\mc L_m}(z)=\sum\limits_{\substack{0\leq k\leq m-2\\ k\text{{ is even}}}}C_{k,m}\mc L_{m-k}\log^k|z|,$$
%$$C_{k,m}=\dfrac{(2m-3)!!2^k(m-2)!(2m-k-3)!}{(2m-2)!!(2m-3)!(k+1)!(m-k-2)!}$$
%where $C_{k,m}$ are some non-zero rational numbers given in \cite{Levin_A}.
Denote by $\Q[\P(\k)]_m$ a vector space freely generated by the points of $\P(\k)$. We denote the generators as $\{a\}_m, a\in\P(\k)$. A. Goncharov \cite{goncharov1995geometry} defined a subspace $\mc R_m(\k)\subset \Q[\P(\k)]_m$. If $\k=\C$ and $\sum n_\alpha \{z_\alpha\}_m\in\mc R_m(\k)$ then
$$\sum\limits_{\alpha}n_\alpha \wt{\mc L}_m(z_\alpha)=0.$$
Conjecturally, the group $\mc R_m$ describes all functional equations for the function $\wt{\mc L}_m$. 
The $m$-th Bloch group is defined by the formula $B_m(F)=\Q[\P(F)]_m/\mc R_m(F)$.
Let $Y$ be an algebraic variety and $f\in\P(\k(Y))$. Define an element in $\L^{2m-1}\k(Y\t(\P)^{m-1})^\t$. If $f=0,\infty$, then this element is equal to $0$. Otherwise,
$$\omega_m(Y,f)=(1-x_1)\wedge x_1\wedge (x_1-x_2)\wedge x_2\wdw (x_{m-2}-x_{m-1})\wedge x_{m-1}\wedge (x_{m-1}-f).$$
This form is closely related to the algebraic cycles representing classical polylogarithm in motivic cohomology \cite{levine1994bloch}. 
In the case $Y=\spec\k$ we will denote this element simply by $\omega_m(f)$. We have $[(\P)^{m-1},\omega_m(a)]=\gamma_{Q_a}$, where $Q_a(x_1,\dots, x_{m-1})=x_1\dots x_{m-1}-a$. We have the following result \cite{Levin_A,goncharov2005polylogarithms}:
\begin{proposition}
\label{prop:formula}
    There is a constant $q_m\in\C^\t$ such that for any  $a\in \P(\C)$, we have
    \begin{equation}
    \tag{***}
        {\mc P}_m((\P)^{m-1},\omega_m(a))=q_m\wt{\mc L_m}(a).
    \end{equation}

\end{proposition}
Define a map $\mc T_m\cl \Q[\P(\k)]_m\to \widehat B_m(\k)$ sending 
    $\{a\}_m$ to  $[(\P)^{m-1}, \omega_m(a)].$

\begin{theorem}
\label{th:functional_equations_classical_polylogarithms}
\label{cor:KZ_conjecture}
    The map $\mc T_m$ sends $\mc R_m$ to zero. In particular, we have the well-defined map $\mc T_m\cl B_m(\k)\to \widehat B_m(\k)$.
\end{theorem}

Informally speaking, this theorem says that any functional equation from $\mc R_m$ for the function $\wt{\mc L}_m$ is a formal corollary of $(*)$ and $(***)$. 

A definition of a period can be found in \cite{kontsevich2001periods}. A famous conjecture stated in loc. cit. says that any relations between periods formally follows from the projection formula, Stokes formula and additivity. We have the following corollary:

\begin{corollary}
    Any relation from $\mc R_m$ for the function $\wt{\mc L}_m$ satisfies the statement of the Kontsevich-Zagier conjecture.
\end{corollary}

%\subsection{Notation} We work with rational coefficients, so any chain complex is a vector space over $\Q$.
%Also we assume that any field has characteristic zero.
%It seems that one can show that the cohomology of the complex $\L$ can be shown to be functorial for arbitrary maps between smooth varieties. However at the moment this is unknown.

%\subsection{Acknowledgements}

\section{Beilinson-Soule vanishing for the complex $\L(\k,m)$}

\subsection{Preliminaries lemmas}
We need the following two lemmas:
\begin{lemma}
\label{lemma:tr.deg.zero}
    If $tr.deg.(f_1,\dots, f_n)<\dim X$, then
    $$[X,f_1\wdw f_n]=0.$$
\end{lemma}
\begin{proof}
    Follows form \cite[Lemma 3.5.]{bol_2024}
\end{proof}

\begin{lemma}
\label{lemma:Galois_descent}
    Let $L/K$ be a finite Galois extension of fields of characteristic zero. Let $x\in(\L^k L^\t)\otimes\Q$. Assume that $x$ is invariant under the Galois group. Then $x$ can be represented as a linear combination of the element of the  form $a_1\wdw a_k, b_1\wedge (1-b_1)\wedge b_3\wdw b_k,$
        where $a_i\in K$  and $b_i\in L$.

\end{lemma}
\begin{proof}
    This follows from the existence of the norm map on Milnor $K$-theory, see \cite{suslin1979reciprocity}.
\end{proof}

\begin{lemma}
\label{lemma:bir}
    Let $Y$ be a proper variety and $\delta\in \L^n\k(Y)^\t$. For any birational automorphism $\theta$ of $Y$, we have:
    $$[Y,\delta]=[Y,\theta^*(\delta)].$$
\end{lemma}

\begin{proof}
    There is a birational morphism $\ph\cl \wt Y\to Y$ such that $\theta\circ\ph$ is regular. So we get
    $[Y, \theta^*(\delta)]=[\wt Y, \ph^*(\theta^*(\delta))]=[\wt Y, (\theta\circ\ph)^*(\delta)]=[Y,\delta].$
\end{proof}

\subsection{The filtration $\mc F$}
We recall that $m$ denotes the motivic weight and $j$ denotes the cohomological degree. The numbers $p$ and $n$ are defined by the formulas $p=m-j, n=2m-j$. The group $\L(\spec\k, m)_j$ is generated by the elements $[Y,a]$ where $Y$ is proper $p$-dimensional variety and $a\in\L^{n}\k(Y)^\t$.
Define a decreasing filtration $\mc F$ on $\L(m)_j$ where $\mc F_l$ is generated by
$$[Y,f_1\wedge (1-f_1)\wdw f_l\wedge(1-f_l)\wedge g_{2l+1}\wdw g_n].$$
The goal of this subsection is to prove the following theorem:
\begin{theorem}
\label{th:about_filtration}
    The group $\L(\k,m)_j$ is equal to $\mc F_{m-j}$.
\end{theorem}
%The Corollary \ref{cor:generators_of_Bloch_group} immediately follows from this theorem.
\begin{lemma}
\label{lemma:about_filtration_2}
Let $Z$ be a proper variety. Let $f_1,\dots, f_l$ be non-zero rational functions on $Z$ and assume that $f_i\ne 1$. Let $g_1,\dots, g_{n-2l}$ be rational functions on $Z\t\P$. We consider $f_i$ as rational functions on $Z\t \P$ via the natural projection $Z\t\P\to Z$. The following element is in $\mc F_{l+1}$:
    $$[Z\t\P, f_1\wedge (1-f_1)\wdw  f_l\wedge (1-f_l)\wedge g_1\wdw g_{n-2l}].$$
\end{lemma}

\begin{proof}
Denote this element by $\xi$. Let $\ph\cl Z'\to Z$ be some alteration. Denote by $\ph'=(\ph, id)$ the corresponding map $Z'\t\P\to Z\t\P$. We can choose $\ph$ such that the rational functions $(\ph')^*(g_i)$ are products of linear functions in $t$ with coefficients in rational functions on $Z'$. So we can assume that $\xi$ is a linear combination of elements of the following form:
$$\xi=[Z'\t\P, f_1\wedge (1-f_1)\wdw f_l\wedge (1-f_l)\wedge (t-g_1')\wdw (t-g_r')\wedge g_{r+1}'\wdw g_{n-2l}'].$$
   In this formula $g_i'$ are some non-zero rational functions on $Z'$. The case $r=0$ follows from lemma \ref{lemma:tr.deg.zero}. Let us assume that $r\geq 1$. Making the change of variable $\wt t=t-g_1'$ we can assume that $g_1'=0$. The proof is by induction on $r$. 
   Let $r=1$. Denote
   $$\alpha=f_1\wedge (1-f_1)\wdw f_l\wedge (1-f_l)\wedge t\wedge g_2'\wdw g_{n-2l}'.$$
Let $\theta$ be an automorphism of $Z'\t\P$ given by the formula $(z,t)\mapsto (z,t^{-1})$. We get
   $$[Z'\t\P,\alpha]=[Z'\t\P,\theta^*(\alpha)]=-[Z'\t\P,\alpha].$$
   And so $\xi=[Z'\t\P,\alpha]=0$.
    Assume that $r>1$. Making the change of variable $\wt t=t/g_2'$, we can assume that $g_2'=1$. In this case the statement is obvious.
\end{proof}

\begin{proof}[The proof of Theorem \ref{th:about_filtration}] Assume that $l<m-j$. We need to check that $\mc F_l=\mc F_{l+1}$. As $m-j=p$ we get $l<p$.
 The vector space $\mc F_l$ is generated by the elements of the following form
    $$\xi=[Y,f_1\wedge (1-f_1)\wdw f_l\wedge (1-f_l)\wedge g_1\wdw g_{n-2l}].$$ 
    By Lemma \ref{lemma:tr.deg.zero}, we can assume that the extension $\k(f_1,\dots, f_l,g_1,\dots, g_{n-2l})\subset \k(Y)$ is finite. Let $s=p - tr.deg.(\k(f_1,\dots, f_l)).$ As $l<p$ this number is positive. Reordering $g_i$, we can assume that $(g_1,\dots, g_s)$ is a transcendence basis of $\k(X)$ over $\k(f_1,\dots, f_l)$. 
    We need to check that $\xi\in\mc F_{l+1}$. Set 
    $$\alpha = f_1\wedge (1-f_1)\wdw f_l\wedge (1-f_l),\quad \beta= g_1\wdw g_{s}.$$

Let $L=\k(f_1,\dots, f_l,g_1,\dots, g_{s})$. We can assume that the extension $L\subset \k(Y)$ is Galois. Otherwise, there is an alteration $\ph\colon Y'\to Y$ such that the composition $L\subset\k(Y)\subset\k(Y')$ is Galois and it is enough to prove the statement for $Y'$.  Denote by $\gamma$ the averaging of the element $g_{s+1}\wdw g_{n-2l}$ under the Galois group of $\k(Y)/L$. By Lemma \ref{lemma:bir} we get
$$\xi=[X,\alpha\wedge \beta\wedge \gamma].$$ 
Set $s'=n-2l-s$. We have $s'\geq 0$. We apply Lemma \ref{lemma:Galois_descent} to $\gamma$. We get that  $\xi$ is a linear combination of the elements of the following form:
$$[Y,\alpha\wedge \beta\wedge h\wedge (1-h)\wedge \gamma'], h\in\k(Y)^\t, h\ne 1, \gamma'\in\Lambda^{(s'-2)}(\k(Y)^\t).$$
$$[Y, \alpha\wedge \beta\wedge h_1'\wdw h_{s'}'], \quad h_i'\in L.$$
The first element obviously belongs to $\mc F_{l+1}$. Let us show that the second element also belongs to $\mc F_{l+1}$. Denote this element by $\eta$. Consider the maps $\psi_1'\cl Y\to (\P)^l, \psi_2\cl Y\to (\P)^s$ given by the formulas $$\psi'_1(z)=(f_1(z),\dots, f_l(z)), \quad\psi_2(z)=(g_1(z),\dots, g_s(z)).$$ We can assume that these two maps are regular. Denote by $Z$ the image of $\psi'_1$ and let $\psi_1$ be the natural map $Y\to Z$. Let $\psi=(\psi_1, \psi_2)\cl Y\to Z\t(\P)^s$. Denote by $y_i, t_i$ the coordinates on $(\P)^l$ and $(\P)^s$. Let $\overline y_i$ be the restriction of $y_i$ to $Z$ and let $\alpha'=\ol y_1\wedge (1-\ol y_1)\wdw \ol y_l\wedge (1-\ol y_l)$. For some rational functions $h_1'',\dots, h_{s'}''$ we get:
$$\eta=\dfrac 1{\deg\psi}[Z\t(\P)^s, \alpha'\wedge t_1\wdw t_s\wedge h_1''\wdw h_{s'}''].$$
This element belongs to $\mc F_{l+1}$ by Lemma \ref{lemma:about_filtration_2}.
\end{proof}

\subsection{The proof of Theorem \ref{th:BS_Lambda}}
 By the previous theorem we know that $\L(\k,m)_j$ is equal to $\mc F_{m-j}=\mc F_p$. As $j\leq 0$ and $j=n-2p$, we get $n\leq 2p$. So $p\geq n/2$. If $p>n/2$ then $\L(\k,m)_j=\mc F_p=0$. So we can assume that $p=n/2$. Let us show that $\mc F_p=0$. Denote
$$\xi = [Y; f_1\wedge (1-f_1)\wdw f_p\wedge (1-f_p)].$$
    If the transcendence degree of the field $\k(f_1,\dots, f_p)$ is strictly smaller than $p$ we get zero by Lemma \ref{lemma:tr.deg.zero}. So the extension $\k(f_1,\dots, f_p)\subset \k(X)$ is finite. We can assume that $f_i$ are regular. Denote by $\psi \colon Y\to (\P)^p$ a map given by the formula $z\mapsto (f_1(z),\dots, f_{p}(z))$. We get $$\xi=\dfrac{1}{\deg\psi}[(\P)^p, t_1\wedge (1-t_1)\wdw t_{p}\wedge (1-t_p)].$$

    In this formula $t_i$ are canonical coordinates on $(\P)^p$. Let us show that this element is equal to zero. Denote $\alpha = t_1\wedge (1-t_1)\wdw t_{p}\wedge (1-t_p)$. Let $\theta$ be an automorphism of $(\P)^p$ given by the formula $\theta(t_1,\dots, t_p)=(1-t_1,t_2,\dots, t_p)$. We get 

    $$[(\P)^p, \alpha]=[(\P)^p, \theta^*(\alpha)]=[(\P)^p, -\alpha]=-[(\P)^p, \alpha].$$
    So $[(\P)^p,\alpha]=0$.
The theorem \ref{th:BS_Lambda} is proved.

\begin{proof}[The proof of Theorem \ref{th:generators_of_Bloch_group}]It follows from Theorem \ref{th:about_filtration} that the group $\widehat B_m(\k)$ is generated by the element of the following form:
$$[Y,f_1\wedge (1-f_1)\wdw f_{m-1}\wedge (1-f_{m-1})\wedge g].$$
Denote this element by $\xi$. Consider the map $\wt\psi_1\cl Y\to (\P)^{m-1}$ given by the formula $y\mapsto (f_1(y),\dots, f_{m-1}(y))$. We can assume that $\wt\psi_1$ is regular. Denote the image of $\wt\psi_1$ by $Z$. Let $\psi_1$ be the natural map $Y\to Z$. Consider several cases:

\begin{enumerate}
  \item     Let us assume that $\dim Z=\dim Y$. In this case the map $\psi_1$ is an alteration. We can assume that the element $g$ lies in  $\k(f_1,\dots, f_{m-1})$. We get:
    $$\xi=\dfrac 1{\deg \psi_1}[(\P)^{m-1}, x_1\wedge (1-x_1)\wdw x_{m-1}\wedge (1-x_{m-1})\wedge Q(x_1,\dots, x_{m-1})].$$
    %This implies the statement of the corollary.
    \item Let us assume that $\dim Z<\dim Y$. Consider the map $\psi_2\cl Y\to\P$ given by the formula $y\mapsto g(y)$. We can assume that $\psi_2$ is regular. Denote by $\psi\cl Y\to Z\t\P$ a map given by the formula $y\mapsto (\psi_1(y), \psi_2(y))$. If the map $\psi$ were not dominant than the transcendence degree of the extension $\k\subset\k(f_1,\dots, f_{m-1},g)$ would be strictly smaller than $m-1$. In this case Lemma \ref{lemma:tr.deg.zero} would imply that $\xi=0$. So we can assume that the map $\psi$ is an alteration. We get
    $$\xi=\dfrac 1{\deg\psi}[Z\t\P,\wt f_1\wedge (1-\wt f_1)\wdw \wt f_{m-1}\wedge (1-\wt f_{m-1})\wedge t].$$
    In this formula $\wt f_i$ are some functions on $Z$ and $t$ is the canonical coordinate on $\P$. By Lemma \ref{lemma:about_filtration_2} this element lies in $\mc F_{m}$ and so is equal to zero. 
\end{enumerate}
\end{proof}
\begin{corollary}
\label{cor:constant_zero}
    If one of the functions $f_i$ is constant, then
    $$[X, f_1\wdw f_{2m-1}]=0.$$
%In particular, for any $c_i\in\k^\t$, we have:
    %$$[X, (c_1f_1)\wdw (c_{2m-1}f_{2m-1})]=[X, f_1\wdw f_{2m-1}].$$
\end{corollary}
\begin{proof}%[The proof of Corollary \ref{cor:constant_zero}]
Define a map $mult\colon \L(\k, m-1)_{0}\t \k^\t\to \L(\k, m)_{1}$ given by the formula
$$[Y, f_1,\dots, f_{2m-2}]\otimes \lambda \mapsto [Y, f_1,\dots, f_{2m-2}, \lambda].$$
    By Theorem \ref{th:BS_Lambda}, the group $\L(\k, m-1)_{0}$ is zero and so the map $mult$ is zero. This implies the statement of the corollary.
\end{proof}

\section{Functional equations for classical polylogarithms}

\subsection{The group $\mc R_m(F)$}
Denote by $\Q[\P(F)]_m$ a free vector space generated by the symbols $\{z\}_m, z\in \P(F)$. Let $(F,\nu)$ be a discrete valuation field. Denote by $\ol F_\nu$ the residue field. For an element $a\in F$ with $\nu(a)=0$ we denote by $\ol a$ the corresponding element in $\ol F_\nu$. Define a map $res_\nu\cl \Q[\P(F)]_m\to \Q[\P(\overline F_\nu)]_m$ as follows. Let $\{h\}_m\in \Q[\P(F)]_m$. If $\nu(h)=0$, then $res_\nu(\{h\}_m)=\{\overline h\}_m$. Otherwise, $res_\nu(\{h\}_m)=0$. Let $Y$ be a curve. For a point $y\in Y$ denote by $\nu_y$ a discrete valuation corresponding to $y$. We write $res_y$ instead of $res_{\nu_y}$. We set $F^\t_\Q=F^\t\otimes_\Z\Q$.
 
 Let us give the definition of the group $\mc R_m(F)$. The definition is by induction on $m$. Let $m=2$. Define a map $\wt \delta_{2,F}\cl \Q[\P(F)]_2\to \L^2 F^\t_\Q$ by the following formula $$
\wt\delta_{2,F}(\{a\}_2)=
\begin{cases}
0 &\text{if } a=0,1,\infty\\
a\wedge (1-a) &\text{otherwise}.
\end{cases}
$$
Let $L=F(t)$. For any $\alpha\in \ker\delta_{2,L}$ we have the   element $res_{0}(\alpha)-res_{\infty}(\alpha)$. This element belongs to the group $ \Q[\P(F)]_2$. The group $\mc R_2(F)$ is generated by $\{0\}_m$ and $\{\infty\}_m$ and by the elements of the form $res_{0}(\alpha)-res_{\infty}(\alpha)$, where $\alpha\in\ker\wt\delta_{2,L}$.

Let $m>2$. We set $B_{m-1}(F)=\Q[\P(F)]_{m-1}/\mc R_{m-1}(F)$. Define a map  $\wt\delta_{m,F}\cl  \Q[\P(F)]_m\to B_{m-1}(F)\otimes_\Z F^\t_\Q$ by the following formula $$
\wt\delta_{m,F}(\{a\}_m)=
\begin{cases}
0 &\text{if } a=0,\infty\\
\{a\}_{m-1}\otimes a &\text{otherwise}.
\end{cases}
$$
The group $\mc R_m(F)$ is generated by $\{0\}_m, \{\infty\}_m$ and by the elements of the form $res_{0}(\alpha)-res_{\infty}(\alpha)$, where $\alpha\in\ker\wt\delta_{m,L}$.

\subsection{The proof of Theorem \ref{th:functional_equations_classical_polylogarithms}}

To prove Theorem \ref{th:functional_equations_classical_polylogarithms} we need several lemmas. Let $X$ be an algebraic variety. Denote by $\L(X,m)'$ the subgroup of $\L(X,m)$ generated by the triples $[Y,a,\ph]$, where $\ph$ is surjective. Denote by $s$ the natural map $\L(X,m)\to\L(\k(X),m)$ corresponding to the restriction to the generic point. We proved in \cite[3.3]{bolbachan2023chow} that the complex $\L(X,m)$ satisfies the localisation sequence. The following lemma is an easy consequence of this fact.

\begin{lemma}
    \label{lemma:restriction_to_generic_point}
        The restriction of $s$ to  $\L(X,m)'$ induces an isomorphism $\L(X,m)'\to \L(\k(X),m)$.
    \end{lemma}

    \begin{proof}
    Denote the map stated in the lemma by $s'$. It is easy to see that $s'$ is surjective.
    Let $U$ be some non-empty open subset of $X$.  Denote by $s_U$ the natural map $\L(X,m)\to \L(U,m)$ and by $s_U'$ the restriction of this map to $\L(X,m)'$. To show that the map $s'$ is injective in is enough to show that for any non-empty $U\subset X$ the map $s_U'$ is injective. 
        
Let $Z=X\bs U$. Let $\L(X,m)_Z$ be the subgroup of $\L(X,m)$ generated by $[Y,a,\psi]$ such that the image of $\psi$ is contained in $Z$. Define a map $r\cl \L(X,m)\to \L(X,m)_Z$ as follows. Let $[Y,a,\psi]\in \L(X,m)$. If the image of $\psi$ is contained in $Z$ then $r([Y,a,\psi])$ is equal to $[Y,a,\psi]$. Otherwise $r([Y,a,\psi])$ is zero. It is easy to see that this map is well defined and that $r$ is left inverse to the natural inclusion $\L(X,m)_Z\to \L(X,m)$.

Let $\xi\in \L(X,m)_Z\cap \L(X,m)'$. As $\xi\in \L(X,m)_Z$ we have $r(\xi)=\xi$. On the other hand as $\xi\in \L(X,m)'$ we get $r(\xi)=0$. This implies that $\L(X,m)_Z\cap \L(X,m)'=\varnothing$.

It follows from \cite[Theorem 3.7.]{bolbachan2023chow} that the kernel of the map $s_U$ is equal to $\L(X,m)_Z$. This implies that the map $s_U'$ is injective.
    \end{proof}
    
     Let $g\cl X_1\to X_2$ be a proper morphism. Set $s=\dim X_2-\dim X_1$. Define a morphism of complexes $\psi_*\cl \L(X_1,m)_*\to \L(X_2, m+s)_{*+2s}$ by the formula $g_*([Y,a,\psi])=[Y,a,g\circ\psi]$. The fact that $g_*$ is a morphism of complexes is clear. Let $X$ be $\P$ and denote by $\pi_X$ the structural map. For a point $x\in X$ denote by $i_x$ the embedding of the point $x$ into $X$. We have $(\pi_X)_*\circ (i_x)_*=id$. Denote by $t$ the canonical coordinate on $X$.

\begin{lemma}
\label{lemma:differential_of_polylogarithms}
Let $m\geq 3$. The following statements are true:
    \begin{enumerate}
        \item Let $a\in \k^\t$. Then $$d([(\P)^{m-1},\omega_m(a)])=[(\P)^{m-2}, \omega_{m-1}(a)\wedge a].$$
        \item Let $h\in\k(X)^\t$. Denote by $\psi_k$ the natural projection $X\t(\P)^{k}\to X$. We have the following equality in the group $L(X, m+1)_3$: 
        \begin{align*}
            &d([X\t(\P)^{m-1},\omega_m(X,h)\wedge t,\psi_{m-1}])=[X\t(\P)^{m-2},\omega_{m-1}(X,h)\wedge h\wedge t,\psi_{m-2}]+\\&(i_\infty)_*\mc T_{m}(h(\infty))-(i_0)_*\mc T_{m}(h(0)).
        \end{align*}
    \end{enumerate}
\end{lemma}
This lemma will be proved in the next subsection.

\begin{proof}[The proof of Theorem \ref{th:functional_equations_classical_polylogarithms}]
    Let us prove by induction on $m$ that the statement of the theorem holds for all fields of characteristic zero. The case $m=2$ follows from \cite[Lemma 6.2]{bol_2024}. So we can assume that $m\geq 3$. 
Let
    $$\alpha=\sum c_{\alpha}\{h_\alpha(t)\}_m\in \ker\wt\delta_{m,\k(t)}.$$
We recall that $X=\P$. Consider the following element:
$$\xi=\sum c_\alpha [X\t(\P)^{m-1},\omega_{m-1}(h_{\alpha}(t))\wedge t, \psi_{m-1}].$$
By Lemma \ref{lemma:differential_of_polylogarithms} we get
$$d((\pi_X)_*(\xi))=(\pi_X)_*(d(\xi))=(\pi_X)_*(\eta)+\sum\limits_{\alpha}c_{\alpha}(\mc T_{m}(h_\alpha(\infty))-\mc T_{m}(h_\alpha(0))),$$
where $$\eta=\sum\limits_{\alpha}c_{\alpha}[X\t(\P)^{m-2},\omega_{m-1}(X,h_{\alpha(t)})\wedge h_{\alpha(t)}\wedge t,\psi_{m-2}].$$
By Theorem \ref{th:BS_Lambda}, we get

$$\sum\limits_{\alpha}c_\alpha\mc T_{m}(h_\alpha(0))-\sum\limits_{\alpha}c_\alpha\mc T_{m}(h_\alpha(\infty))=(\pi_X)_*(\eta).$$
To prove the statement of the theorem, it is enough to check  that $\eta=0$. To do this, by Lemma \ref{lemma:restriction_to_generic_point}, it is enough to check that $s(\beta)=0$. Denote by $\mb P_{\k(t)}^1$ the projective line over the field $\k(t)$.  We get
$$s(\beta)=\sum\limits_{\alpha}c_{\alpha}[(\mb P_{\k(t)}^1)^{m-2},\omega_{m-1}(h_{\alpha(t)})\wedge h_{\alpha(t)}\wedge t].$$
Define a map $\wt{\mc T}_{m-1}^{(2)}\cl B_{m-1}(\k(t)\otimes \L^2(\k(t)^\t_\Q)\to \L(\k(t),m+1)_3$ by the following formula
$$\wt{\mc T}_{m-1}^{(2)}(\{a\}_{m-1}\otimes b\wedge c)=\begin{cases}
0 &\text{if } a=0,\infty\\
 [(\mb P_{\k(t)}^1)^{m-2},\omega_{m-1}(a)\wedge b\wedge c] &\text{otherwise}.
\end{cases}$$
    This map is well-defined by the inductive assumption for the field $\k(t)$.
We get
$$s(\beta)=\mc P_{m-1}^{(2)}(\wt\delta_m(\xi)\wedge t)=0.$$
In the last formula we have used the fact that $\alpha\in \ker \wt\delta_m$.

\end{proof}

%\begin{corollary}
%    There is a canonical morphism of complexes $\Gamma(\k,m)\to \L(\k,m)$
%\end{corollary}

%\begin{proof}
%    This map is defined by the formula $\{a\}_m\otimes a_1\wdw a_k\mapsto [(\P){m-1}, \mc \omega_m(a),a_1,\dots,a_k]$. It is clear that this map is a morphism of complexes.
%\end{proof}

\begin{proof}[The proof of Corollary \ref{cor:KZ_conjecture}]
We need to check that the relations $(*)$ and $(***)$  satisfy the statement of Kontcevich-Zagier conjecture.
    For $(*)$ this is clear. For $(***)$ this follows the analysis of the proofs given in \cite{Levin_A,goncharov2005polylogarithms}.
\end{proof}

\subsection{The proof of Lemma \ref{lemma:differential_of_polylogarithms}}
\begin{lemma}
\label{lemma:contracted_lies_on_two_divisors}
Let $Y$ be a smooth variety, $\psi\cl \wt Y\to Y$ an alteration and $D\subset \wt Y$ a divisor contracted under $\ph$. Let $f_1,\dots, f_n$ are non-zero rational functions on $Y$. Set $H=\psi(D)$ and $\alpha=f_1\wdw f_n$. Assume that $[D,\psi^*(\alpha)]\ne 0$. Then there are $i\ne j$ such that $H$ lies on both divisors $\supp((f_i))$ and $\supp((f_j))$.
\end{lemma}
\begin{proof}
    If for any $i$, $\ord_D(f_i)=0$, then $[D,\ph^*(\alpha)]=0$. So for some $i$, we have $\ord_D (f_i)\ne 0$. We can assume that $i=1$. Assume by contradiction that for any $j\ne 1 $, we have $\ord_D(f_j)=0$. Then $\ts_D(\ph^*(\alpha))=\nu_D(\ph^*(f_1))\wedge \res{\ph^*(f_2\wdw f_n)}{D}$. Since $D$ is contracted under $\ph$, the transcendence degree of the field $\k(\res{\ph^*(f_2)}{D},\dots, \res{\ph^*(f_n)}{D})$ is strictly smaller than $\dim D$. So by Lemma \ref{lemma:tr.deg.zero}, we have $[D,\ph^*(\alpha)]=0$. A contradiction.
\end{proof}

\begin{lemma}
\label{lemma:xi:ijk}
    Let $Y$ be a smooth $m$-dimensional variety, $\ph\cl \wt Y\to Y$  an alteration and $D$ a divisor on $\wt Y$. Let $f_1,\dots, f_{2m}$ be non-zero rational functions on $Y$. Set $\alpha=f_1\wdw f_{2m}$. Assume that for some $1\leq i_1<\dots <i_k\leq 2m$ and some $n_2,\dots, n_k\in\Z$, the restriction of the rational function $\ph^*(f_{i_1}f_{i_2}^{n_2}\dots f_{i_k}^{n_k})$ to $D$ is  constant. Then $[D, \ts_D(\ph^*(\alpha))]=0.$
\end{lemma}

\begin{proof}
    We can assume that $(i_1,\dots i_k)=(1,\dots, k)$. Let $g=f_{1}f_{2}^{n_2}\dots f_{k}^{n_k}$. We have: $\alpha = g\wedge f_2\wdw f_{2m}$. We know that the restriction of $\ph^*(g)$ to $D$ is constant. So for some $c\in\k^\t$ we have $\ts_D(\ph^*(\alpha))=c\wedge \ts_D(\ph^*(f_2\wdw f_{2m}))$. The statement follows from Corollary \ref{cor:constant_zero}.  
\end{proof}

\begin{proof}[The proof of lemma \ref{lemma:differential_of_polylogarithms}]\begin{enumerate}
    \item Let $x_1,\dots, x_{m-1}$ be coordinates on $(\P)^{m-1}$. Let $y_1=x_1$ and for $i>1$ we set $y_i=x_i/x_{i-1}$. In the new coordinates $y_i$ we get:
$$[(\P)^{m-1},\omega_m(a)]=[(\P)^{m-1}, \wt\omega_m(a)],$$
where
$$\wt\omega_m(a)=(1-y_1)\wedge y_1\wdw (1-y_{m-1})\wedge y_{m-1}\wedge (y_1\wdw y_{m-1}-a).$$
We set $\alpha = \wt\omega_m(a)$.
Let $\ph\colon \wt Y\to (\P)^{m-1}$ be a proper birational morphism such that $$d([Y,\alpha])=\sum\limits_{D\subset \wt Y}[D,\ts_D(\ph^*(\alpha))].$$
Let us show that for any $D$ contracted under $\ph$ we get $[D,\ts_D(\ph^*(\alpha))]=0$. Denote $H=\ph(D)$ and $\eta_D = [D,\ts_D(\ph^*(\wt\omega_m(a)))]$.

By lemma \ref{lemma:contracted_lies_on_two_divisors} we know that $H$ lies on the divisor of one of the functions $y_i, 1-y_i$ for some $i$. We apply Lemma \ref{lemma:xi:ijk}. Consider several cases:
\begin{enumerate}
    \item $H$ lies on the zeros of $y_i$. In this case the restriction of $\phi^*(1-y_i)$ to $D$ is equal to $1$ and so $\eta_D=0$.
    \item $H$ lies on the zeros of $1-y_i$. This case is similar.
    \item $H$ lies on the poles of $y_i$. In this case the restriction of $\phi^*(y_i/(1-y_i))$ to $D$ is equal to $-1$ and so $\eta_D=0$.
\end{enumerate}

So we have proved that $\eta_D=0$. This implies that
$$d([Y,\alpha])=\sum\limits_{D\subset (\P)^{m-1}}[D,\ts_D(\alpha)].$$
The only non-trivial term corresponds to the divisor $\{y_1\dots y_{m-1}=a\}$. The corresponding expression is given in the statement of the lemma.
\item

Let $Y=X\t(\P)^{m-1}$ and
$$\alpha=(1-y_1)\wedge y_1\wdw (1-y_{m-1})\wedge y_{m-1}\wedge (y_1\wdw y_{m-1}-h(t))\wedge t\in \L^{2m}\k(Y)^\t.$$
Denote by $\pi_Y$ the canonical projection $Y\to X$. We have
$$[Y,\omega_m(X,h(t))\wedge t,\pi_Y]=[Y,\alpha,\pi_Y].$$
Choose some proper birational morphism $\ph\cl \wt Y\to Y$, such that
$$d([Y, \alpha, \pi_Y])=\sum\limits_{D\subset \wt Y}[D, \ts_D(\beta), \pi_Y\circ\ph \circ j_D].$$
In this formula $\beta=\ph^*(\alpha)$ and $j_D$ is the canonical embedding $D\to \wt Y$. Set $\eta_D=[D, \ts_D(\beta), \pi_Y\circ\ph \circ j_D]$. Let us show that for any $D$ contracted under $\ph$ we have $\eta_D=0$.

Denote by $\pi_D$ the natural projection $D\to X$.  We need to show that $\eta_D=0$.

 Let $H$ be the image of $D$ under $\ph$. Denote by $\pi_H$ the natural projection $H\to X$. If $\pi_H$ is dominant, then by the previous item and Lemma \ref{lemma:restriction_to_generic_point} we get $\eta_D=0$. So we can assume that $\pi_H$ is constant. Denote the image of $\pi_H$ by $w\in X$. We have $(i_w)_*(\pi_*(\eta_D))=\eta_D$. So it is enough to check that $\eta'_D=\pi_*(\eta_D)$ is zero. Set $\beta = \ph^*(\alpha)$. We have $\eta'_D=[D,\ts_D(\beta)]$.

 Assume that $w\ne 0,\infty$. Then the restriction of $\ph^*(t)$ to $D$ is constant and so by Lemma \ref{lemma:xi:ijk} we have $\eta'_D=0$. So $w\in\{0,\infty\}$. We will consider only the case $w=0$, the other case is similar.

Assume that $H$ is contained in the divisor of one of the functions $y_i,1-y_i$. In this case, similarly to the proof of the previous item, one can show that $\eta_D'=0$. So, by Lemma \ref{lemma:contracted_lies_on_two_divisors}, we can assume that $H$ is contained in the divisor $\{y_1\dots y_{m-1}=h(t)\}$. Denote $f=y_1\dots y_{m-1}$ and $g=f-h(t)$. We apply Lemma \ref{lemma:xi:ijk}. Consider several cases:
\begin{enumerate}
    \item $h(0)=0$. In this case the restriction of $\ph^*(g/f)$ to $D$ is equal to $1$.
    \item $h(0)\in\k^\t$. In this case the restriction of $\ph^*(f)$ to $D$ is equal to $h(0)\in\k^\t$.
    \item $h(0)=\infty$. We write $h(t)=\wt h(t)t^{-r}$ for some function $\wt h$, such that $\wt h(0)\in \k^\t$. The restriction of $\ph^*(gt^r)$ to $D$ is equal to $-\wt h(0)\in\k^\t$ and so $\eta'_D=0$.
\end{enumerate}
We conclude that
$$d([Y, \alpha, \pi_Y])=\sum\limits_{D\subset Y}[D, \ts_D(\alpha), \pi_Y \circ j_D].$$
The expression $[D, \ts_D(\alpha), \pi_Y \circ j_D]$ is non-zero only if $D$ is equal to one of the following divisors: $\{y_1\dots y_{m-1}=h(t)\}, \{t=0\}, \{t=\infty\}$. The corresponding terms given in the statement of the lemma.

\end{enumerate}
\end{proof}

\bibliographystyle{alpha} % mathematics and physical sciences
 % APS-like style for physics
\bibliography{mylib} 
Vasily Bolbachan\\
Skolkovo Institute of Science and Technology, 3 Nobelya Str., 121205 Moscow, Russia;\\ HSE University, HSE-Skoltech International Laboratory of Representation
Theory and Mathematical Physics, 119048, Russia, Moscow, Usacheva str., 6\\
\emph{E-mail:} \texttt{vbolbachan{\fontfamily{ptm}\selectfont @}gmail.com}
\end{document}